\documentclass[12pt,reqno]{amsart}
\usepackage{etoolbox}

   \makeatletter

 \patchcmd{\@setaddresses}{\scshape\ignorespaces}{\ignorespaces}{}{} 

\appto\maketitle{%
\let\@makefnmark\relax  \let\@thefnmark\relax
\ifx\@empty\addresses\else\@footnotetext{%
  \vskip-\bigskipamount\@setaddresses}
  }
\def\enddoc@text{}
\makeatother

\makeatletter
\patchcmd\maketitle
  {\uppercasenonmath\shorttitle}
  {}
  {}{}
\patchcmd\maketitle
  {\@nx\MakeUppercase{\the\toks@}}
  {\the\toks@}
  {}
  {}{}
\patchcmd\@settitle{\uppercasenonmath\@title}{\Large}{}{}
\patchcmd\@setauthors
  {\MakeUppercase{\authors}}
  {\authors}
  {}{}
\makeatother
\usepackage{amsmath,amssymb,amsthm}
\usepackage{color}
\usepackage{url}
\usepackage{tikz-cd}
\usepackage[utf8]{inputenc}
\usepackage[T1]{fontenc}
\newtheorem{theorem}{Theorem}[section]

\newtheorem{corollary}{Corollary}[section]

\newtheorem{lemma}{Lemma}[section]
\newtheorem{remark}{Remark}[section]

\numberwithin{equation}{section}

\usepackage[colorlinks=true]{hyperref}
\hypersetup{urlcolor=blue, citecolor=red , linkcolor= blue}
  \newtheorem{thqt}{Theorem}

  \newtheorem{propqt}[thqt]{Proposition}
  \newtheorem{dfn}[thqt]{Definition}

\newcommand{\h}{\mathcal{H} }
\newcommand{\bh}{\mathbb B(\h)}

\newcommand{\bad}{\mathbb{B}_{A^{1/2}}(\h)}

\newcommand{\wat}{W_A(T)}
\newcommand{\maxat}{W_{\max}^A(T)}
\newcommand{\modu}[1]{\left\vert#1\right\vert}
\newcommand{\nor}[1]{\left\Vert#1\right\Vert}
\newcommand{\nora}[1]{\left\Vert#1\right\Vert_{A}}
\newcommand{\set}[1]{\{#1\}}
\newcommand{\cit}{\mathbb{C}}
\newcommand{\dt}{\delta_{T}}

\begin{document}

\author[A. Baghdad, E. H. Benabdi and K. Feki] {\large{Abderrahim Baghdad }$^{1_{a,b}}$, \large{El Hassan Benabdi }$^{2}$ and \large{Kais Feki }$^{3_{a,b}}$}
\title[A note on the generalized maximal numerical range of operators]{A note on the generalized maximal numerical range of operators}

\address{$^{[1]}$ Department of Mathematics, Faculty of Sciences-Semlalia, University Cadi Ayyad, Marrakesh, Morocco.}
\email{\url{bagabd66@gmail.com}}

\address{$^{[2]}$ Department of Mathematics, Laboratory of Mathematics, Statistics and Applications, Faculty of Sciences,
Mohammed V University in Rabat, Rabat, Morocco.}
\email{\url{e.benabdi@um5r.ac.ma}}

\address{$^{[3_a]}$ University of Monastir, Faculty of Economic Sciences and Management of Mahdia, Mahdia, Tunisia}
\address{$^{[3_b]}$ Laboratory Physics-Mathematics and Applications (LR/13/ES-22), Faculty of Sciences of Sfax, University of Sfax, Sfax, Tunisia}
\email{\url{kais.feki@hotmail.com}\,;\,\url{kais.feki@fsegma.u-monastir.tn}}



\date{\today}

\maketitle

\begin{abstract}
The paper considers some new properties of the so-called $A$-maximal numerical range of operators, denoted by $W_{\max}^A(\cdot)$, where $A$ is a positive bounded linear operator acting on a complex Hilbert space $\mathcal{H}$. Some characterizations of $A$-normaloid operators are also given. In particular, we extend a recent recent by Spitkovsky in [Oper. Matrices, 13, 3(2019)]. Namely, it is shown that an $A$-bounded linear operator $T$ acting on $\mathcal{H}$ is $A$-normaloid if and only if $W_{\max}^A(T)\cap \partial W_A(T)\neq\varnothing$. Here $\partial W_A(T)$ stands for the boundary of $A$-numerical range of $T$. Some new $A$-numerical radius inequalities generalizing and improving earlier well-known results are also given.
\end{abstract}

\section{Introduction and Preliminaries}\label{s1}
Throughout this work $\mathcal{H}$ stands for a non trivial complex Hilbert space with inner product $\langle\cdot,\cdot\rangle$ and associated norm $\|\cdot\|$. By $\mathbb{B}(\mathcal{H})$, we denote the $C^*$-algebra of bounded linear operators acting on $\mathcal{H}$ with the identity operator $I_\mathcal{H}$ (or simply $I$ if no confusion arises). For simplicity, by an operator we mean an operator in $\mathbb{B}(\mathcal{H})$. For every operator $T$, its adjoint is denoted by $T^*$, its range by $\mathcal{R}(T)$ and its null space by $\mathcal{N}(T)$.

 For the sequel, the following facts are useful. An operator $T$ is said to be positive if $\langle Tx, x\rangle\geq 0$ for every $x\in \mathcal{H}$. By $\mathbb{B}(\mathcal{H})^+$, we denote the cone of positive (semi-definite) operators, i.e.
$$\mathbb{B}(\mathcal{H})^+=\left\{T\in \mathbb{B}(\mathcal{H})\,;\,\langle Tx, x\rangle\geq 0,\;\forall\;x\in \mathcal{H}\;\right\}.$$
For the rest of this article, we suppose that $A\in\mathbb{B}(\mathcal{H})^+$ is a nonzero operator which clearly defines the following positive semidefinite sesquilinear form:
$$\langle\cdot,\cdot\rangle_{A}:\mathcal{H}\times \mathcal{H}\longrightarrow\mathbb{C},\;(x,y)\longmapsto\langle x, y\rangle_{A} :=\langle Ax, y\rangle=\langle A^{1/2}x, A^{1/2}y\rangle.$$
Here $A^{1/2}$ means the square root of $A$. We denote by ${\|\cdot\|}_A$ the seminorm induced by ${\langle \cdot, \cdot\rangle}_A$ which is given by ${\|x\|}_A = \sqrt{{\langle x, x\rangle}_A}=\sqrt{\|A^{1/2}x\|}$ for every $x\in\mathcal{H}$. It can be checked that ${\|x\|}_A = 0$ if and only if $x\in \mathcal{N}(A)$. So, ${\|\cdot\|}_A$ is a norm on $\mathcal{H}$ if and only if $A$ is one-to-one. Furthermore, one may verify that the semi-Hilbert space $(\mathcal{H}, {\|\cdot\|}_A)$ is complete if and only if $\mathcal{R}(A)$ is closed in $(\mathcal{H},\|\cdot\|)$. For a given $T\in\mathbb{B}(\mathcal{H})$, if there exists $c>0$ such that $\|Tx \|_{A} \leq c \|x \|_{A}$ for all $x\in
\overline{\mathcal{R}(A)}$, then it holds:
\begin{equation*}\label{semii}
\|T\|_A:=\sup_{\substack{x\in \overline{\mathcal{R}(A)},\\ x\not=0}}\frac{\|Tx\|_A}{\|x\|_A}=\displaystyle\sup_{\substack{x\in \overline{\mathcal{R}(A)},\\ \|x\|_A= 1}}\|Tx\|_{A}<\infty.
\end{equation*}
If $A=I$, we get the classical norm of an operator $T$ which will be simple denoted by $\|T\|_A$. From now on, we denote $\mathbb{B}^{A}(\mathcal{H}):=\left\{T\in \mathbb{B}(\mathcal{H})\,;\,\|T\|_{A}< \infty\right\}$. It is important to note that $\mathbb{B}^{A}(\mathcal{H})$ is not generally a subalgebra of $\mathbb{B}(\mathcal{H})$ (see \cite{feki01}). Further, it is difficult to check that $\|T\|_A=0$ if and only if $ATA=0$. Recently, there are many papers that study operators defined on a semi-Hilbert space $(\mathcal{H}, {\|\cdot\|}_A)$. One may see \cite{bakfeki01, bakna, bakna2, ,book,feki03,kz, majsecesuci} and their references.

Let $T\in \bh$. An operator $S\in\bh$ is called an $A$-adjoint operator of $T$ if ${\langle Tx, y\rangle}_{A} = {\langle x, Sy\rangle}_{A}$ for all $x, y\in \mathcal{H}$ (see \cite{acg1}). Clearly, $S$ is an $A$-adjoint of $T$ if and only if $AS=T^*A$, i.e., $S$ is a solution in $\bh$ of the equation $AX= T^*A$. We mention here that this type of operator equations can be studied by using the following famous theorem due to Douglas (for its proof see \cite{doug}).
\begin{thqt}\label{dougt}
If $T, U \in \bh$, then the following statements are equivalent:
\begin{itemize}
\item[{\rm (i)}] $\mathcal{R}(U) \subseteq \mathcal{R}(T)$,
\item[{\rm (ii)}] $TS=U$ for some $S\in \bh$,
\item[{\rm (iii)}] There exists  $\lambda> 0$ such that $\|U^*x\| \leq \lambda\|T^*x\|$ for all $x\in \h$.
\end{itemize}
If one of these conditions holds, then there exists a unique solution of the operator equation $TX=U$, denoted by $Q$, such that $\mathcal{R}(Q) \subseteq \overline{\mathcal{R}(T^{*})}$. Such $Q$ is called the reduced solution of $TX=U$.
\end{thqt}
Let $\bad$ denote the set of all operators that admit $A^{1/2}$-adjoints. An application of Theorem \ref{dougt} shows that
$$\mathbb{B}_{A^{1/2}}(\mathcal{H})=\left\{T \in \mathbb{B}(\mathcal{H})\,;\;\exists \,\lambda > 0\,\text{ such that }\|Tx\|_{A} \leq \lambda \|x\|_{A},\;\forall\,x\in \mathcal{H}  \right\}.$$
If $T\in \mathbb{B}_{A^{1/2}}(\mathcal{H})$, then $T$ is said $A$-bounded. It can be observed that if $T\in \mathbb{B}_{A^{1/2}}(\mathcal{H})$, then $T(\mathcal{N}(A))\subseteq \mathcal{N}(A)$. Further, the following property $\|TS\|_A\leq \|T\|_A\|S\|_A$ holds for all $T,S\in \mathbb{B}_{A^{1/2}}(\mathcal{H})$. Also, if $T\in\mathbb{B}_{A^{1/2}}(\mathcal{H})$, then the authors of \cite{fg} showed that
\begin{align*}
\|T\|_A
&=\sup\left\{\|Tx\|_{A}\,;\;x\in \mathcal{H},\,\|x\|_{A}= 1\right\}\\
&=\sup\left\{|\langle Tx, y\rangle_A|\,;\;x,y\in \mathcal{H},\,\|x\|_{A}=\|y\|_{A}= 1\right\}.
\end{align*}
For more details regarding the class of $A$-bounded operators, we refer the reader to \cite{acg3,feki01,majsecesuci} and the references therein. Note that $\mathbb{B}_{A^{1/2}}(\mathcal{H})$ is a subalgebra of $\mathbb{B}(\mathcal{H})$ which is neither closed nor dense in $\mathbb{B}(\mathcal{H})$. Moreover, the following inclusions:
\begin{equation}\label{inc}
\mathbb{B}_{A^{1/2}}(\mathcal{H})\subseteq
\mathbb{B}^{A}(\mathcal{H})\subseteq \mathbb{B}(\mathcal{H})
\end{equation}
hold. Note that in general the inclusions in \eqref{inc} are proper. However, if $A$ is an injective operator, then obviously $\mathbb{B}_{A^{1/2}}(\mathcal{H})=
\mathbb{B}^{A}(\mathcal{H})$. Further, if $A$ has a closed range in $\mathcal{H}$, then it can be seen that $\mathbb{B}^{A}(\mathcal{H})= \mathbb{B}(\mathcal{H})$. So, the inclusions in \eqref{inc} remain equalities if $A$ is injective and has a closed range. We refer to \cite{acg1,acg2,acg3,feki01} and the references therein for an account of results related the theory of semi-Hilbert spaces.

The notion of the maximal numerical range induced by a positive operator $A$ has recently been introduced by Baklouti et al. in \cite{bakfeki01}. More precisely, we have the following definition.
\begin{dfn}
Let $T\in \mathbb{B}^A(\mathcal{H})$. The $A$-maximal numerical range of $T$, denoted by $W_{\max}^A(T)$, is defined as
\begin{align*}
W_{\max}^A(T)
&=\{\lambda\in \mathbb{C}\,;\,\exists\,(x_n)\subseteq \mathcal{H}\,;\,\|x_n\|_A=1,\lim_{n}\langle T x_n, x_n\rangle_A=\lambda,\\
&\phantom{++++++++++}\;\text{and}\;\displaystyle\lim_{n}\|Tx_n\|_A=\|T\|_A\;\}.
\end{align*}
\end{dfn}
\noindent
For every $T\in\bh$, it was shown in \cite{bakfeki01} that  $W_{\max}^A(T)$ is non-empty, convex and compact subset of $\mathbb{C}$.\\ Notice that the notion of the maximal numerical range of an operator $T\in\mathbb{B}(\mathcal{H })$, denoted by $W_{\max}(T)$ (that is when $A=I$; the identity operator), was first introduced by Stampfli in \cite{s1}, in order to determine the norm of the inner derivation acting on $\bh$. Recall that the inner derivation $\dt$ associated with $T\in \bh$  is defined by $$\dt:\bh\longrightarrow \bh,~X\longmapsto TX-XT.$$
For this, in the same paper \cite{s1}, the author first  established  the following.
\begin{thqt}\label{equi1970}
Let $T\in \mathbb{B}(\mathcal{H})$. Then the following conditions are equivalent:
\begin{itemize}
  \item  [(1)] $0\in W_{\max}(T)$,
  \item [(2)] $\|T\|^2+|\lambda|^2\leq \|T+\lambda\|^2$  for any $\lambda\in \mathbb{C}$,
  \item [(3)] $\|T\|\leq \|T+\lambda\|$ for any $\lambda\in \mathbb{C}$.
\end{itemize}
\end{thqt}
\noindent
Here $T+\lambda$ is denoted to be $T+\lambda I$ for any $\lambda \in \cit$.
\begin{corollary}\label{stamcoro}  Let $T \in \bh$. Then,  there is a unique scalar $	c_{T}$ such that
$$  \Vert T - c_{T}\Vert^2 +  \vert \lambda \vert^2 \leq   \Vert (T- c_{T}) - \lambda  \Vert^2,~ \text{for all}~    \lambda  \in  \mathbb{C}. $$
Moreover, $ 0\in W_{\max}(T)$ if and only if $c_{T}=0.$
\end{corollary}
\noindent
The scalar $c_{T}$ is called the center of mass of $T$. Note that $c_{T}$ is the unique scalar
  satisfying the following
\begin{equation*}
\Vert T-c_{T} \Vert =\inf_{\lambda \in \cit}  \Vert T- \lambda   \Vert.
\end{equation*}
The scalar $\nor{T-c(T)}$ is denoted by $d_{A}(T)$ and is called the distance of $T$ to scalars.
The author    \cite{s1} then proved that for any $T \in \bh$
$$ \nor{\dt} =2d(T).$$
\noindent
Recall that an operator $T\in \mathbb{B}(\mathcal{H})$ is said to be normaloid if $\omega(T)=\|T\|$, where $\omega(T)$ is denoted to be the numerical radius of $T$ which is given by
$$\omega(T)=\sup\{|\lambda|\,;\;\lambda\in W(T) \}.$$
Here $ W(T)$ is denoted to be the numerical range of $T$ and it is defined by Toeplitz in \cite{t1} as
$$W(T):=\{\langle T x, x\rangle;\;x \in \mathcal{H}\;\;\text{with}\;\|x\|=1\}.$$
Equivalent condition is $r(T)=\nor{T}$, see, \cite{Gustafson}. Here, $r(T)$ is the spectral radius of $T$. Recently, Spitkovsky in \cite{spitkovsky} gave the following characterization of a normaloid operator.
\begin{thqt}\label{spitkovsky}
Let $T\in \mathbb{B}(\mathcal{H})$. Then the following conditions are equivalent:
\begin{itemize}
  \item  [(1)] $T$ is a normaloid operator,
  \item [(2)] $W_{\max}(T)\cap \partial W(T)\neq\varnothing$.
\end{itemize}
\end{thqt}
\noindent
Here $\partial L$ stands for the boundary of a subsest $L$ in the complex plane.
\vspace{0.3cm}

\noindent
Notions of the numerical range and numerical radius are generalized in \cite{bakfeki01} as follows.
\begin{dfn}
Let $T \in \bh$. The $A$-numerical range and the $A$-numerical radius of $T$ are respectively given by
$$W_A(T):=\{\langle T x, x\rangle_A\,;\;x \in \mathcal{H}\;\;\text{with}\;\|x\|_A=1\}, $$
and
$$\omega_A(T):=\sup\{|\lambda|\,;\;\lambda\in W_A(T)\}.$$
\end{dfn}
\noindent
It is important to mention that $\omega_A(T)$ may be equal to $+ \infty$ for some $T\in \mathcal{B}(\mathcal{H})$ (see \cite{feki01}). However, $\omega_A(\cdot)$ defines a seminorm on $\bad$ which is equivalent to $\nor{T}_{A}$. More precisely, for any $T\in \bad$, we have
\begin{equation}\label{inequa}
\frac{1}{2}\nora{T}\leq \omega_{A}(T) \leq \nor{T}_{A},
\end{equation}
see \cite{bakfeki01}.\\
Recently, the concept of $A$-normaloid operators is  introduced by the third author in \cite{feki01} as follows.
\begin{dfn}
An operator $T\in \mathbb{B}_{A^{1/2}}(\mathcal{H})$ is said to be $A$-normaloid if $r_A(T)=\|T\|_A$, where
$$r_A(T)=\displaystyle\lim_{n}\|T^n\|_A^{\frac{1}{n}}.$$
\end{dfn}
\noindent
Some characterizations of $A$-normaloid operators are proved in \cite{feki01}. In particular, we have the following proposition.
\begin{propqt}[\cite{feki01}]\label{charactenormailoid}
Let $T\in \mathbb{B}_{A^{1/2}}(\mathcal{H})$. Then, the following assertions are equivalent:
\begin{itemize}
  \item [(1)] $T$ is $A$-normaloid,
  \item [(2)] $\|T^n\|_A=\|T\|_A^n$ for all positive integer $n$,
  \item [(3)] $\omega_A(T)=\|T\|_A$,
    \item [(4)] There exists a sequence $(x_n)\subseteq\mathcal{H}$ such that $\|x_n\|_A=1$.
    \begin{equation*}
    \lim_{n}\|T x_n\|_A=\|T\|_A \;\text{  and  }\; \lim_{n}|\langle T x_n, x_n\rangle_A|=\omega_A(T).
    \end{equation*}
\end{itemize}
\end{propqt}

\noindent
Our aim in this work is to give some new characterizations of $A$-normaloid operators. Mainly, by considering the operator range $\mathcal{R}(A^{1/2})$ endowed with its canonical Hilbertian structure, which will be denoted by $\mathbf{R}(A^{1/2})$, and then using the connection between $A$-bounded operators and operators acting on the Hilbert space $\mathbf{R}(A^{1/2})$, we extend Theorem \ref{spitkovsky} to the context of semi-Hilbert spaces. Moreover, several new properties concerning the $A$-maximal numerical range of $A$-bounded operators are established. One main target of this article is to generalize Theorem \ref{equi1970} for $T\in \bad$. In addition, we give a sufficient and necessary condition for which the $A$-center of mass of an operator $T\in \mathbb{B}_{A^{1/2}}(\mathcal{H})$ belongs to $W_{\max}^A(T)$. Other properties are also studied.

In the sequel, if $T$ is any operator in $\bad$, we define
$$\Gamma_A(T):=\big\{z\in \mathbb{C}\,;\;|z|=\|T\|_A \big\}.$$

\section{Main Results}\label{s2}
We begin this section with the following theorem which gives another useful characterization of $A$-normaloid operators. We will denote by $\overline{L}$ the closure of any subset $L$ in the complex plane.
\begin{theorem}\label{newcarac}
Let $T\in\mathbb{B}_{A^{1/2}}(\mathcal{H})$. Then,
\begin{itemize}
  \item [(1)] $T$ is $A$-normaloid,
  \item [(2)] $\Gamma_A(T)\cap \overline{W_A(T)}\neq \varnothing$.
\end{itemize}
\end{theorem}
\begin{proof}
$(1)\Rightarrow(2)$: Assume that $T$ is $A$-normaloid. Then, by Proposition \ref{charactenormailoid} we have $\omega_A(T)=\|T\|_A$. So, there exists a sequence $(z_n)\subseteq W_A(T)$ such that $\displaystyle\lim_{n}|z_n|=\|T\|_A$. By compactness of  $\overline{W_A(T)}$ we can,  taking a subsequence of $(z_n)$ if  needed,  assume that $(z_n)$ converges to some $z\in \overline{W_A(T)}$. Therefore, $|z|=\|T\|_A$, so $z\in \Gamma_A(T)\cap \overline{W_A(T)}$.\\
$(2)\Rightarrow(1)$: Let $z\in \Gamma_A(T)\cap \overline{W_A(T)}$. We have $\omega_{A}(T) \geq \modu{z}=\nor{T}_{A}$. From  Inequalities \eqref{inequa}, we deduce that $\omega_{A}(T) =\nora{T}$. That is,  $T$ is $A$-normaloid.
\end{proof}

\noindent

\noindent
Now, we aim to generalize Theorems \ref{equi1970} and \ref{spitkovsky} for $T\in \bad$. To accomplish this goal, some facts from \cite{acg3} should be recalled. Let $X=\mathcal{H}/\mathcal{N}(A)$ be the quotient space of $\mathcal{H}$ by $\mathcal{N}(A)$. It can be observed that $\langle\cdot,\cdot\rangle_A$ induces on $X$ the following inner product:
$$[\overline{x},\overline{y}]=\langle x, y\rangle_A=\langle Ax, y\rangle,$$
for every $\overline{x},\overline{y}\in X$. We note that $(X,[\cdot,\cdot])$ is not complete unless $\mathcal{R}(A)$ is a closed subspace in $\mathcal{H}$. However, de Branges et al. proved in \cite{branrov} (see also \cite{f4}) that the completion of $X$ under the inner product $[\cdot,\cdot]$ is isomorphic to the Hilbert space $\mathcal{R}(A^{1/2})$ endowed with the following inner product:
$$(A^{1/2}x,A^{1/2}y):=\langle Px, Py\rangle,\;\forall\, x,y \in \mathcal{H},$$ where $P$ stands for the orthogonal projection of $\mathcal{H}$ onto the closure of $\mathcal{R}(A)$.\\
From now on, the Hilbert space $\left(\mathcal{R}(A^{1/2}), (\cdot,\cdot)\right)$ will be simply denoted by $\mathbf{R}(A^{1/2})$. Further, the symbol $\|\cdot\|_{\mathbf{R}(A^{1/2})}$ represents the norm induced by $(\cdot,\cdot)$. It is crucial to note that $\mathcal{R}(A)$ is dense
in $\mathbf{R}(A^{1/2})$ (see \cite{feki01}). Since $\mathcal{R}(A)\subseteq \mathcal{R}(A^{1/2})$, then we see that
\begin{equation}\label{usefuleq001}
( Ax,Ay)= (A^{1/2} A^{1/2}x,A^{1/2}A^{1/2}y)= \langle PA^{1/2}x, PA^{1/2}y\rangle=\langle x , y\rangle_A,\quad\forall\,x,y\in \mathcal{H},
\end{equation}
whence,
\begin{equation}\label{usefuleq01}
\|Ax\|_{\mathbf{R}(A^{1/2})}=\|x\|_A,
\end{equation}
for any $x \in \h$. For more information concerning the Hilbert space $\mathbf{R}(A^{1/2})$, the interested reader is referred to \cite{acg3}. \\
Let us consider now the operator $Z_A$ defined by:
$$Z_A:\h \longrightarrow \mathbf{R}(A^{1/2}),~x\longmapsto Z_Ax=Ax.$$
Further, the following useful proposition is stated in \cite{acg3}.
\begin{propqt}\label{prop_arias}
Let $T\in \mathbb{B}(\mathcal{H})$. Then $T\in \mathbb{B}_{A^{1/2}}(\mathcal{H})$ if and only if there exists a unique $\widehat{T}\in \mathbb{B}(\mathbf{R}(A^{1/2}))$ such that $Z_AT =\widehat{T}Z_A$.
\end{propqt}
\noindent
Before we move on, it is important to state the following lemmas. The proof of the first one can be found in  \cite{feki01}.
\begin{lemma}\label{imporeq2009}
Let $T\in \mathbb{B}_{A^{1/2}}(\mathcal{H})$. Then
\begin{itemize}
  \item [(i)] $\|T\|_A=\|\widehat{T}\|_{\mathbb{B}(\mathbf{R}(A^{1/2}))}$.
  \item [(ii)] $\omega_A(T)=\omega(\widehat{T})$.
\end{itemize}
\end{lemma}

\begin{lemma}\label{lemmamaximal}
Let $T\in\mathbb{B}_{A^{1/2}}(\mathcal{H})$. Then
\begin{equation*}
W_{\max}^A(T)=W_{\max}(\widehat{T}),
\end{equation*}
where $\widehat{T}$ is the operator given by Proposition \ref{prop_arias}.
\end{lemma}
\begin{proof}
We have $Z_AT =\widehat{T}Z_A$, that is, $ATx=\widehat{T}Ax$ for all $x\in \mathcal{H}$. Now, let $\lambda\in W_{\max}^A(T)$, then there exists $(x_n)\subseteq \mathcal{H}$ such that $\|x_n\|_A=1$,
$$\lim_{n\to+\infty}\langle T x_n, x_n\rangle_A=\lambda,\text{ and }\;\displaystyle\lim_{n\to+\infty}\|Tx_n\|_A=\|T\|_A.$$
Set $y_n=Ax_n\in \mathcal{R}(A^{1/2})$. By using \eqref{usefuleq001} together with \eqref{usefuleq01}, we have $\|y_n\|_{\mathbf{R}(A^{1/2})}=\|x_n\|_A=1$ and
$$\langle T x_n, x_n\rangle_A=(AT x_n,Ax_n)=(\widehat{T} y_n,y_n),$$
Again, by \eqref{usefuleq01}, we infer that
$$\|Tx_n\|_A=\|ATx_n\|_{\mathbf{R}(A^{1/2})}=\|\widehat{T} y_n\|_{\mathbf{R}(A^{1/2})}.$$
On the other hand, by Lemma \ref{imporeq2009} we have $\|T\|_A=\|\widehat{T}\|_{\mathbb{B}(\mathbf{R}(A^{1/2}))}$. This implies that $\lambda\in W_{\max}(\widehat{T})$ and so $W_{\max}^A(T)\subseteq W_{\max}(\widehat{T})$. Conversely, let $\lambda\in W_{\max}(\widehat{T})$, then there exists $(y_n)\subseteq \mathcal{R}(A^{1/2})$ such that $\|y_n\|_{\mathbf{R}(A^{1/2})}=1$,
$$\lim_{n\to+\infty}(\widehat{T}y_n,y_n)=\lambda,\text{ and }\;\displaystyle\lim_{n\to+\infty}\|\widehat{T}y_n\|_{\mathbf{R}(A^{1/2})}=\|\widehat{T}\|_{\mathbb{B}(\mathbf{R}(A^{1/2}))}=\|T\|_A.$$
Since $(y_n)\subseteq \mathcal{R}(A^{1/2})$ for all $n$, then there exists $(x_n)\subseteq \mathcal{H}$ such that $y_n=A^{1/2}x_n$. So, $\|A^{1/2}x_n\|_{\mathbf{R}(A^{1/2})}=1$,
\begin{equation}\label{doublelim}
\lim_{n\to+\infty}(\widehat{T}A^{1/2}x_n,A^{1/2}x_n)=\lambda \text{ and }\;\displaystyle\lim_{n\to+\infty}\|\widehat{T}A^{1/2}x_n\|_{\mathbf{R}(A^{1/2})}=\|T\|_A.
\end{equation}
On the other hand, since $\mathcal{R}(A)$ is dense in $\mathbf{R}(A^{1/2})$, then for any $n\in \mathbb{N}$, there exists $(x_{n,k})\subseteq \mathcal{H}$ such that
\begin{equation*}
\lim_{k\to +\infty}\|Ax_{n,k}-A^{1/2}x_n\|_{\mathbf{R}(A^{1/2})}=0.
\end{equation*}
This gives
\begin{equation}\label{eqlilm}
\lim_{k\to +\infty}\|Ax_{n,k}\|_{\mathbf{R}(A^{1/2})}=1.
\end{equation}
Moreover, by \eqref{doublelim} we have
$$\lim_{n,k\to+\infty}(\widehat{T}Ax_{n,k},Ax_{n,k})=\lambda \text{ and }\;\displaystyle\lim_{n,k\to+\infty}\|\widehat{T}Ax_{n,k}\|_{\mathbf{R}(A^{1/2})}=\|T\|_A.$$
Let $z_k=\dfrac{x_{n,k}}{\|Ax_{n,k}\|_{\mathbf{R}(A^{1/2})}}$. So, by using \eqref{eqlilm}, we obtain
$$\lim_{k\to+\infty}(\widehat{T}Az_k,Az_k)=\lambda \text{ and }\;\displaystyle\lim_{k\to+\infty}\|\widehat{T}Az_k\|_{\mathbf{R}(A^{1/2})}=\|T\|_A.$$
On the other hand, we have
$$(\widehat{T}Az_k,Az_k)=(ATz_k,Az_k)\;\text{ and }\;\|\widehat{T}Az_k\|_{\mathbf{R}(A^{1/2})}=\|ATz_k\|_{\mathbf{R}(A^{1/2})}.$$
So, by applying \eqref{usefuleq001} together with \eqref{usefuleq01}, we infer that
$$\lim_{k\to+\infty}\langle T z_k, z_k\rangle_A=\lambda \text{ and }\;\displaystyle\lim_{k\to+\infty}\|Tz_k\|_A=\|T\|_A.$$
Furthermore, $\|Az_k\|_{\mathbf{R}(A^{1/2})}=\|z_k\|_A=1$. So, we deduce that $\lambda\in W_{\max}^A(T)$. Hence, the proof is complete.
\end{proof}
\noindent
Now, we are in a position to prove the following three theorems. The first one has been proved in \cite{bakfeki01}, however we can obtain the same result  as an immediate consequence of Lemma \ref{lemmamaximal} and \cite[Lemma 2]{s1}.
\begin{theorem}
Let $T\in\mathbb{B}_{A^{1/2}}(\mathcal{H})$. Then $W_{\max}^A(T)$ is  convex.
\end{theorem}

\begin{theorem}\label{maintheorem01}
Let $T\in \mathbb{B}_{A^{1/2}}(\mathcal{H})$. Then the following conditions are equivalent:
\begin{itemize}
 \item [(1)] $0\in W_{\max}^A(T)$.
  \item [(2)] $\|T\|_A^2+|\lambda|^2\leq \|T+\lambda\|_A^2$  for any $\lambda\in \mathbb{C}$.
  \item [(3)] $\|T\|_A\leq \|T+\lambda\|_A$ for any $\lambda\in \mathbb{C}$.
\end{itemize}
\end{theorem}
\begin{proof}
Note first that by using Theorem \ref{equi1970}, we obtain the equivalence between the following assertions:
\begin{itemize}
  \item [(i)] $0\in W_{\max}(\widehat{T})$.
  \item [(ii)] $\|\widehat{T}\|_{\mathbb{B}(\mathbf{R}(A^{1/2}))}^2+|\lambda|^2\leq \|\widehat{T}+\lambda I_{\mathbf{R}(A^{1/2}}\|_{\mathbb{B}(\mathbf{R}(A^{1/2}))}^2$  for any $\lambda\in \mathbb{C}$.
  \item [(iii)] $\|\widehat{T}\|_{\mathbb{B}(\mathbf{R}(A^{1/2}))}\leq \|\widehat{T}+\lambda I_{\mathbf{R}(A^{1/2}}\|_{\mathbb{B}(\mathbf{R}(A^{1/2}))}$ for any $\lambda\in \mathbb{C}$.
\end{itemize}
On the other hand, by Lemma \ref{lemmamaximal}, we have $W_{\max}^A(T)=W_{\max}(\widehat{T})$. Moreover, by Lemma \ref{imporeq2009}, we have $\|T\|_A=\|\widehat{T}\|_{\mathbb{B}(\mathbf{R}(A^{1/2}))}$. Also, notice that $T+\lambda \in \mathbb{B}_{A^{1/2}}(\mathcal{H})$ for any $\lambda \in \cit$ since $ \mathbb{B}_{A^{1/2}}(\mathcal{H})$ is a subalgebra of $\bh$. Then, from Proposition \ref{prop_arias}, for any $\lambda \in \cit$ there exists a unique
 $\widehat{T+\lambda}\in \mathbb{B}(\mathbf{R}(A^{1/2}))$ such that $Z_A(T+\lambda) =\widehat{T+\lambda}Z_A$. So, all what remains to prove is that $\|T+\lambda\|_A=\|\widehat{T}+\lambda I_{\mathbf{R}(A^{1/2}}\|_{\mathbb{B}(\mathbf{R}(A^{1/2}))}$ for any $\lambda \in \cit$. But the above equality follows by applying Lemma \ref{imporeq2009} (i) together with the fact that $\widehat{T+\lambda}=\widehat{T}+\lambda I_{\mathbf{R}(A^{1/2}}$ (see \cite{f4}).
\end{proof}
\noindent
Now, we state the third theorem which generalizes Theorem \ref{spitkovsky} for $A$-bounded operators.  We need the following lemma.
\begin{lemma}\label{lemmanew23112019}
Let $T\in\mathbb{B}_{A^{1/2}}(\mathcal{H})$. Then,
$$\Gamma_A(T)\cap W_{\max}^A(T)=\Gamma_A(T)\cap \overline{W_A(T)}.$$
\end{lemma}
\begin{proof}
Since $W_{\max}^A(T)\subseteq\overline{W_A(T)}$ then the first inclusion holds. Now, let $\lambda\in\Gamma_A(T)\cap \overline{W_A(T)}$. Then, $\lambda=\|T\|_A$ and there exists a sequence $(\lambda_n)\subseteq W_A(T)$ such that $\displaystyle \lambda=\lim_{n}\lambda_n$. So, there is a sequence $(x_n)\subseteq \mathcal{H}$ such that $\|x_n\|_A=1$ and $\lambda_n=\langle Tx_n, x_n\rangle_A$ for all $n$. 	By applying the Cauchy-Schwarz inequality we get
\begin{align*}
|\langle Tx_n, x_n\rangle_A|
& =|\langle A^{1/2}Tx_n, A^{1/2}x_n\rangle| \\
 &\leq \|Tx_n\|_A\|x_n\|_A \\
 &=\|Tx_n\|_A\\
 &\leq \|T\|_A.
\end{align*}
So, $\displaystyle\lim_{n}\|Tx_n\|_A=\|T\|_A$. Hence, $\lambda\in\Gamma_A(T)\cap W_{\max}^A(T)$.
\end{proof}
\noindent

Now, we are able to prove one of our main results of this article. We will denote by $\stackrel{\circ}{L}$ the interior of any subset $L$ in the complex plane.
\begin{theorem}
Let $T\in \mathbb{B}_{A^{1/2}}(\mathcal{H})$. Then, the following statements are equivalent
\begin{itemize}
  \item [(1)] $T$ is an $A$-normaloid operator,
  \item [(2)] $W_{\max}^A(T)\cap \partial W_A(T)\neq\varnothing$.
\end{itemize}
\end{theorem}
\begin{proof}
$(1)\Rightarrow(2)$: Assume that $T$ is an $A$-normaloid operator. Then, by applying Theorem \ref{newcarac} together with Lemma \ref{lemmanew23112019}, we get $$\Gamma_A(T)\cap \overline{W_A(T)}=\Gamma_A(T)\cap W_{\max}^A(T)\neq \varnothing.$$
So, there exist $z\in \Gamma_A(T)\cap \overline{W_A(T)}$. Thus, $z$ must lie on the boundary of $W_A(T)$. Since $z$ is also in $W_{\max}^A(T)$, then $W_{\max}^A(T)\cap \partial W_A(T)\neq\varnothing$ as required.\\
$(2)\Rightarrow(1)$: Assume that  $W_{\max}^A(T)\cap \partial W_A(T)\neq\varnothing$.  Notice that in view of Lemma \ref{imporeq2009} we have $T$ is $A$-normaloid if and only if $\widehat{T}$ is a normaloid operator on the Hilbert space $\mathbf{R}(A^{1/2})$. So, in order to prove $(1)$, it suffices to show that
$$W_{\max}(\widehat{T})\cap \partial W(\widehat{T})\neq\varnothing.$$
It was shown in \cite{feki01} that $\overline{W(\widehat{T})}= \overline{W_A(T)}$. Hence $\partial \overline{W(\widehat{T})}=\partial \overline{W_A(T)}$. It is well known that if $C$ is a convex subset in the complex plane, then $\stackrel{\circ}{C}= \stackrel{\circ}{\overline{C}}$. Thus $\partial C= \overline{C}\setminus\stackrel{\circ}{C}= \overline{\overline{C}}\setminus\stackrel{\circ}{\overline{C}}=\partial  \overline{C}$.  Therefore, since both of $W(\widehat{T})$ and $W_A(T)$ are convex,   the equality  $\partial \overline{W(\widehat{T})}=\partial \overline{W_A(T)}$ implies   $\partial W(\widehat{T})=\partial W_A(T)$. Moreover, $ W_{\max}^A(T)=W_{\max}(\widehat{T})$ by Lemma \ref{lemmamaximal}.
We deduce that $W_{\max}(\widehat{T})\cap \partial W(\widehat{T})\neq\varnothing$. This completes the proof.
\end{proof}
\begin{remark}\label{remark2}
\normalfont{ In \cite{chanchan}, the authors gave the following characterization in terms of the numerical radius of a normaloid operator. An operator $T \in\bh$ is  normaloid  if and only if $\omega(T)=\omega_{max}(T)$. Here,  $\omega_{\max}(T)$ is the maximal numerical radius defined by
$$\omega_{\max}(T):=\sup\{|\lambda|\,;\;\lambda\in W_{\max}(T) \}.$$
Therefore, by Lemmas \ref{imporeq2009} and \ref{lemmamaximal}, we can easily obtain the following analogous characterization of $A$-normaloid operators as follows. Notice that this characterization has been also proved by the third author in \cite{feki01}. However, our approach here is different from that used in \cite{feki01}. }
\end{remark}
\begin{theorem}
Let $T\in \mathbb{B}_{A^{1/2}}(\mathcal{H})$. Then, the following statements are equivalent
\begin{itemize}
  \item [(1)] $T$ is an $A$-normaloid operator,
  \item [(2)] $ \omega_{A}(T) = \omega^{A}_{\max}(T), $
\end{itemize}
where $\omega^{A}_{\max}(T)$ is the $A$-maximal numerical radius defined by
 $$\omega^{A}_{\max}(T):=\sup\set{\modu{\lambda};~\lambda \in \maxat}.$$
\end{theorem}

On the other hand, by a similar argument as in the proof of Theorem \ref{maintheorem01} and using Corollary \ref{stamcoro}, we obtain the following corollary.
\begin{corollary}\label{kbb}
Let $T \in \bad$. Then,  there is a unique scalar $c_{A}(T)$ such that
\begin{equation}\label{pythastam}
\Vert T - c_{A}(T)\Vert_{A}^2 +  \vert \lambda \vert^2 \leq   \Vert (T- c_{A}(T)) - \lambda  \Vert_{A}^2,~ \text{for all}~    \lambda  \in  \mathbb{C}.
\end{equation}
Moreover, $ 0\in W_{\max}^{A}(T)$ if and only if $c_{A}(T)=0.$
\end{corollary}
\noindent
Note that $c_{A}(T)=c_{\widehat{T}}$; center of mass of $\widehat{T}$. We call $c_{A}(T)$ the $A$-center of mass of $T$ and we denote $d_{A}(T)=\nor{T-c_{A}(T)}_{A}$ that we call the $A$-distance of $T$ to scalars. Clearly, $c_{A}(T)$ is the unique scalar satisfying
$$ d_{A}(T)=\inf_{\lambda \in \cit}\nor{T-\lambda}_{A}.$$
In the following, we give a formula for $d_{A}(T)$, where $T \in \bad$.
\begin{theorem}\label{formula}
Let $T \in \bad$. Then,
$$ d^{2}_{A}(T)=\sup_{\nor{x}_{A}=1} \left\lbrace  \nor{Tx}^{2}_{A} - \modu{\langle Tx,x \rangle_{A}}^{2}\right\rbrace. $$
\end{theorem}

\begin{proof}
For any $x\in \h$ with $\nor{x}_{A}=1$, we have
\begin{align*}
  d^{2}_{A}(T)=\nor{ T-c_{A}(T)}_{A} ^{2} & \geq
  \nor{ (T-c_{A}(T))x }_{A}^{2}\\
   &=\nor{ Tx}_{A}^{2} + \modu{c_{A}(T) }^{2} - 2Re(\overline{c_{A}(T)} \langle Tx,x  \rangle_{A})   \\
 &\geq \nor{ Tx}_{A}^{2}- \modu{\langle Tx,x  \rangle_{A}}^{2}+\modu{ c_{A}(T) - \langle Tx,x \rangle_{A})}^{2}\\
 &\geq \nor{ Tx}_{A}^{2}- \modu{\langle Tx,x  \rangle_{A}}^{2}.
 \end{align*}
 Whence,
 $$d^{2}_{A}(T)\geq \sup_{\nor{x}_{A}=1} \left\lbrace  \nor{Tx}^{2}_{A} - \modu{\langle Tx,x \rangle_{A}}^{2}\right\rbrace.$$
 Conversely,
 $$ \nor{T-c_{A}(T)}_{A}=\inf_{\lambda \in \cit}\nor{T-\lambda}_{A}=\inf_{\lambda \in \cit}\nor{(T- c_{A}(T))-\lambda}_{A}. $$
 Then, $\nor{T-c_{A}(T)}_{A}\leq \nor{(T- c_{A}(T))-\lambda}_{A}$ for any $\lambda \in \cit$. Since $T-c_{A}(T) \in \bad$, from Theorem \ref{maintheorem01} we get $0\in W_{\max}^A(T-c_{A}(T)) $. So, there exists a sequence $(x_n)\subseteq\h$ with $\nor{x_{n}}_{A}=1$ such that
$$  \lim_{n}\langle (T- c_{A}(T)) x_n, x_n\rangle_A=0 \quad \text{and}  \quad  \lim_{n}\nor{(T- c_{A}(T)) x_n }_{A}=\nor{T- c_{A}(T)}_{A}.  $$
Then, $ \displaystyle\lim_{n}\langle T x_n, x_n\rangle_A = c_{A}(T) $ and
\begin{align*}
  \nor{T- c_{A}(T)}^{2}_{A}&= \lim_{n}\nor{(T- c_{A}(T)) x_n }^{2}_{A}\\
  &= \lim_{n}\left\lbrace  \nor{ Tx_{n}}_{A}^{2}- \modu{\langle Tx_{n},x _{n} \rangle_{A}}^{2}+\modu{ c_{A}(T) - \langle Tx_{n},x_{n} \rangle_{A})}^{2}\right\rbrace\\
 &=\lim_{n}\left\lbrace  \nor{ Tx_{n}}_{A}^{2}- \modu{\langle Tx_{n},x _{n} \rangle_{A}}^{2}\right\rbrace\\
 &\leq \sup_{\nor{x}_{A}=1} \left\lbrace  \nor{Tx}^{2}_{A} - \modu{\langle Tx,x \rangle_{A}}^{2}\right\rbrace.
 \end{align*}
Consequently,
$$d^{2}_{A}(T)= \sup_{\nor{x}_{A}=1} \left\lbrace  \nor{Tx}^{2}_{A} - \modu{\langle Tx,x \rangle_{A}}^{2}\right\rbrace.$$
The proof is complete.
\end{proof}
\begin{remark}\label{remark1} \normalfont{Let $T \in \bad$. There is a sequence $(x_n)\subseteq \h$ with $\nor{x_{n}}_{A}=1$ such that $ \displaystyle\lim_{n}\langle T x_n, x_n\rangle_A = c_{A}(T) $.  We derive that $ c_{A}(T) \in \overline{\wat}$. However,  $ c_{A}(T) $ need not be contained in  $ \maxat$. Indeed, the following corollary gives   sufficient and necessary conditions to have $c_{A}(T) \in \maxat$.}
\end{remark}
\begin{corollary}[\bf Pythagorean Relation]\label{Phytag}
Let $T \in \bad$. Then, the following statements are equivalent:
\begin{itemize}
  \item  [(1)]$c_{A}(T) \in \maxat$,
  \item [(2)] $d^{2}_{A}(T)+ \modu{c_{A}(T)}^{2}= \nor{T}^{2}_{A}$.
\end{itemize}
\end{corollary}
\begin{proof}
$(1) \Rightarrow (2)$:
Assume that  $c_{A}(T) \in \maxat$. There is a sequence $(x_n)\subseteq\h$ with $\nor{x_{n}}_{A}=1$ such that
$$  \lim_{n}\langle T x_n, x_n\rangle_A=c_{A}(T) \quad \text{and}  \quad  \lim_{n}\nor{T x_n }_{A}=\nor{T}_{A}.  $$
As above, we have
\begin{align*}
  \nor{T- c_{A}(T)}^{2}_{A}&\geq \lim_{n}\nor{(T-c_{A}(T))x_{n}}^{2}_{A}\\
 &=\lim_{n}\left\lbrace  \nor{ Tx_{n}}_{A}^{2}- \modu{\langle Tx_{n},x _{n} \rangle_{A}}^{2}\right\rbrace\\
 &= \nor{T}^{2}_{A}-\modu{c_{A}(T)}^{2}.
  \end{align*}
  Hence,
  \begin{equation*}
  \nor{T- c_{A}(T)}^{2}_{A}+\modu{c_{A}(T)}^{2}\geq \nor{T}^{2}_{A}.
  \end{equation*}
  Taking $\lambda =- c_{A}(T)$ in Inequality \eqref{pythastam}, we obtain
  \begin{equation}\label{ineq}
  \nor{T- c_{A}(T)}^{2}_{A}+\modu{c_{A}(T)}^{2}\leq \nor{T}^{2}_{A}.
  \end{equation}
Hence, $$ \nor{T- c_{A}(T)}^{2}_{A}+\modu{c_{A}(T)}^{2}= \nor{T}^{2}_{A}.$$
$(2) \Rightarrow (1)$: Assume that $ d^{2}_{A}(T)+\modu{c_{A}(T)}^{2}= \nor{T}^{2}_{A}$. From the proof of Theorem \ref{formula}, there is a sequence $(x_n)\subseteq\h$ with $\nor{x_{n}}_{A}=1$ such that
 $ \displaystyle\lim_{n}\langle T x_n, x_n\rangle_A = c_{A}(T) $ and
 \begin{align*}
 d^{2}_{A}(T)= \nor{T- c_{A}(T)}^{2}_{A}
 &=\lim_{n}\left\lbrace  \nor{ Tx_{n}}_{A}^{2}- \modu{\langle Tx_{n},x _{n} \rangle_{A}}^{2}\right\rbrace\\
 &=\lim_{n}  \nor{ Tx_{n}}_{A}^{2}- \modu{c_{A}(T)}^{2}.
 \end{align*}
Remembering the hypothesis, we infer that $\lim_{n}  \nor{ Tx_{n}}_{A}=\nor{T}_{A}$. Consequently, $c_{A}(T) \in \maxat$.
\end{proof}
\begin{remark}\label{remark2} \normalfont {Let $T \in \bad$. From Remark \ref{remark1},  $ c_{A}(T) \in \overline{\wat}$. So, $\modu{ c_{A}(T)} \leq \omega_{A}(T)  $. We know that $\maxat \subseteq \overline{\wat}$, the following  question arises: what about $\modu{ c_{A}(T)}$ and $\omega^{A}_{\max}(T)$?}
\end{remark}
\noindent
Define $$m^{A}_{\max}(T):=\inf\set{\modu{\lambda};~\lambda \in \maxat},$$ for any $T\in \bad$. The following answers this question.
\begin{theorem}
Let $T\in \bad$  Then, $$\modu{ c_{A}(T)}\leq m^{A}_{\max}(T). $$ In particular, $$\modu{ c_{A}(T)}\leq \omega^{A}_{\max}(T). $$
\end{theorem}
\begin{proof}
By an argument of compactness, there exists $\alpha \in \maxat$ such that $\modu{\alpha}= m^{A}_{\max}(T)$. Hence, there is a sequence     $(x_{n}) \subseteq \h$ with $\nor{x_{n}}_{A}=1$ satisfying
\begin{equation*}
\alpha =\lim_{n}\langle Tx_{n},x_{n}  \rangle_{A} \quad \text{and} \quad \lim_{n}\nor{ Tx_{n} }_{A} =\nor{T}_{A}.
\end{equation*}
Therefore, we have
\begin{align*}
  \nor{ T-c_{A}(T)}_{A} ^{2} & \geq
  \nor{ (T-c_{A}(T))x_{n} }_{A}^{2}\\
   &=\nor{ Tx_{n}}_{A}^{2} + \modu{c_{A} }^{2} - 2Re(\overline{c_{A}(T)} \langle Tx_{n},x_{n}  \rangle_{A})   \\
 &\geq \nor{ Tx_{n}}^{2}+ \modu{c_{A}(T)}^{2}-2\modu{ c_{A}}  \modu{\langle Tx_{n},x_{n}  \rangle_{A}}.
 \end{align*}
  It results that
 \begin{align}\label{2.3}
  \nor{ T-c_{A}(T)}_{A} ^{2}  & \geq \nor{T}_{A}^{2}+ \modu{ c_{A}(T)}^{2}-2\modu{ c_{A}(T)} m_{\max}^{A}(T)\\
  &= \nor{T}_{A}^{2} - (m_{\max}^{A}(T))^{2} + (m_{\max}^{A}(T)- \modu{c_{A}(T)})^{2}.   \nonumber
  \end{align}
 Thus,
 $$  \nor{ T-c_{A}(T)}_{A} ^{2}+ (m_{\max}^{A}(T))^{2}   \geq  \nor{T}_{A}^{2}  + (m_{\max}^{A}(T)- \modu{c_{A}(T)})^{2}. $$
  We see that
 \begin{equation}\label{d+w' sup norA}
 \nor{ T-c_{A}(T)}_{A} ^{2}+ (m_{\max}^{A}(T))^{2}   \geq  \nor{T}_{A}^{2}
\end{equation}
 and from Inequality   \eqref{ineq}, we get $ m_{\max}^{A}(T) \geq \modu{c_{A}(T)} $.
 \end{proof}
\begin{remark}\normalfont{In \cite{drag1}, it is proved that
\begin{equation}\label{drag}
 \nor{T}^{2} \leq d^{2}(T)+\omega^{2}(T)
\end{equation}
for any $T\in \bh$. From Inequality \eqref{d+w' sup norA}, we have
\begin{equation}\label{kbb}
 \nor{T}_{A}^{2}\leq d_{A}^{2}(T) ^{2}+ (m_{\max}^{A}(T))^{2} \leq  d_{A}^{2}(T)+\omega_{A}^{2}(T).
\end{equation}
Note that, taking $A=I$, Inequality \eqref{kbb} is a refinement of Inequality \eqref{drag}.
On the other hand, from Inequality \eqref{2.3}, we have
$$  \nor{T}_{A}^{2}+ \modu{ c_{A}(T)}^{2} \leq d_{A}^{2}(T)   +2\modu{ c_{A}(T)} m_{\max}^{A}(T). $$
Then
$$ 2\nor{T}_{A}\modu{ c_{A}(T)}\leq d_{A}^{2}(T)   +2\modu{ c_{A}(T)} m_{\max}^{A}(T).  $$
Consequently, if $ c_{A}(T) \neq 0$ (i.e., $0\notin \maxat$), then
$$ \nor{T}_{A} \leq m_{\max}^{A}(T)+\dfrac{1}{2}\dfrac{d_{A}^{2}(T)}{\modu{ c_{A}(T)}}. $$
Therefore, if $ c_{A}(T) \neq 0$, we get}
$$\nor{T}_{A} \leq \inf \left\lbrace  \big(d_{A}^{2}(T) ^{2}+ (m_{\max}^{A}(T))^{2} \big)^{1/2} \ , \ m_{\max}^{A}(T)+ \dfrac{1}{2}\dfrac{d_{A}^{2}(T)}{\modu{ c_{A}(T)}} \right\rbrace .$$
Note that if $c_{A}(T)=0$, then $\nor{T}_{A}=d_{A}(T)$.
\end{remark}
\noindent{\bf{Conflict of interest:}}
 On behalf of all authors, the corresponding author states that
there is no conflict of interest.

\noindent{\bf{Data availability:}}
Data sharing not applicable to the present paper as no data sets were generated or analyzed
during the current study.


\begin{thebibliography}{99} 
\footnotesize

\bibitem{acg1}{M. L. Arias, G. Corach, M. C. Gonzalez,} {Partial isometries in semi-Hilbertian spaces,} Linear Algebra Appl. 428 (7) (2008), 1460-1475.

\bibitem{acg2} {M. L. Arias, G. Corach, M. C. Gonzalez,} {Metric properties of projections in semi-Hilbertian spaces,} Integral Equations and Operator Theory, 62  (2008), pp.11-28.

\bibitem{acg3} {M. L. Arias, G. Corach, M. C. Gonzalez,} {Lifting properties in operator ranges,} Acta Sci. Math. (Szeged) 75:3-4 (2009), 635-653.

\bibitem{AFN} N. Altwaijry, K. Feki, N. Minculete, Further inequalities for the weighted numerical radius of operators, Mathematics (2022) to appear.


\bibitem{bakfeki01}{H. Baklouti, K. Feki, O. A. M. Sid Ahmed,} {Joint numerical ranges of operators in semi-Hilbertian spaces,}  Linear Algebra Appl. 555 (2018), 266-284.


\bibitem{bakna} H. Baklouti, S. Namouri, Closed operators in semi-Hilbertian spaces, Linear Multilinear Algebra (2021) \url{https://doi.org/10.1080/03081087.2021.1932709}.

\bibitem{bakna2} H. Baklouti, S. Namouri, Spectral analysis of bounded operators on semi-Hilbertian spaces, Banach J. Math. Anal. {\bf 16}, 12 (2022).

\bibitem{book} P. Bhunia, S. S. Dragomir, M. S. Moslehian and K. Paul, \textit{Lectures on numerical radius inequalities}, Infosys Science Foundation Series in Mathematical Sciences. Springer, 2022.
    



\bibitem{branrov}{L. de Branges and J. Rovnyak,} {Square Summable Power Series,} Holt, Rinehert and Winston, New York, 1966.


\bibitem{chanchan}{J-T. Chan, K. Chan,} {An observation about normaloid operators,} operators and matrices, Volume 11, Number 3 (2017), 885-890.


\bibitem{doug}{R. G. Douglas,} {On majorization, factorization and range inclusion of operators in Hilbert space,} Proc. Amer. Math. Soc. 17 (1966), 413--416.


\bibitem{drag1} S. S. Dragomir, Inequalities for  the norm and the numerical radius of linear operators in Hilbert spaces, Demonstratio Math.  40(2) (2007), 411--417.

\bibitem{feki01} {K. Feki}, {Spectral radius of semi-Hilbertian space operators and its applications}, Ann. Funct. Anal. 11  (2020), 929-946. \url{https://doi.org/10.1007/s43034-020-00064-y}.


\bibitem{feki03} K. Feki, Some $A$-spectral radius inequalities for $A$-bounded Hilbert space operators, Banach J. Math. Anal. \textbf{16}, 31 (2022)

\bibitem{fg}{M. Faghih-Ahmadi, F. Gorjizadeh,} {$A$-numerical radius of $A$-normal operators in semi-Hilbertian spaces,} Ital. J. Pure Appl. Math. 36, 73--78 (2016)

\bibitem{f4} K. Feki, On tuples of commuting operators in positive semidefinite inner product spaces,  Linear Algebra Appl. \textbf{603}, 313-328 (2020)

\bibitem {Gustafson} {K. E. Gustafson, D. K. M.  Rao},   {Numerical range: The Field of Values of Linear Operators and Matrices}, New York,  NY, USA, (1997).





\bibitem{kz} F. Kittaneh and A. Zamani, \textit{Bounds for $\mathbb{A}$-numerical radius based on an extension of $A$-Buzano inequality}, Journal of Computational and Applied Mathematics, 2023. https://doi.org/10.1016/j.cam.2023.115070






\bibitem{majsecesuci}{W. Majdak , N. A. Secelean, L. Suciu,} {Ergodic properties of operators in some semi-Hilbertian spaces,} Linear and Multilinear Algebra,
61:2 (2013), 139--159.






\bibitem{spitkovsky} I. Spitkovsky, A note on the maximal numerical range, operators and matrices, Volume 13, Number 3 (2019), 601--605

\bibitem{s1}{J. G. Stampfli,} {The norm of derivation,} Pacific J. Math. 33 (1970),  737--747.


\bibitem{t1}{O. Toeplitz,} {Das algebraische Analogou zu einem satze von fejer,} Math. Zeit, 2 (1918), 187--197.



 \end{thebibliography}
\end{document}